\documentclass[pdflatex,sn-mathphys-num]{sn-jnl}

\usepackage{faktor}
\usepackage{soul}

\usepackage{tikz}
\usepackage{tikz-cd}
\usepackage{mathrsfs}

\usepackage{graphicx}%
\usepackage{multirow}%
\usepackage{amsmath,amssymb,amsfonts}%
\usepackage{amsthm}%
\usepackage{mathabx}

\usepackage{enumerate}

\usepackage{mathrsfs}%
\usepackage[title]{appendix}%
\usepackage{xcolor}%
\usepackage{textcomp}%
\usepackage{manyfoot}%
\usepackage{booktabs}%
\usepackage{algorithm}%
\usepackage{algorithmicx}%
\usepackage{algpseudocode}%
\usepackage{listings}%

\newcommand{\cbb}{\mathbb{C}}
\newcommand{\rbb}{\mathbb{R}}
\newcommand{\zbb}{\mathbb{Z}}
\newcommand{\qbb}{\mathbb{Q}}
\newcommand{\tbb}{\mathbb{T}}
\newcommand{\sbb}{\mathbb{S}}
\newcommand{\fcal}{\mathcal{F}}
\newcommand{\hcal}{\mathcal{H}}

\newcommand{\sing}[1]{\left\lbrace #1 \right\rbrace}

\DeclareMathOperator{\codim}{codim}

\let\codim\relax
\DeclareMathOperator{\codim}{codim}
\DeclareMathOperator{\im}{im}

\theoremstyle{thmstyleone}%
\newtheorem{theorem}{Theorem}[section]
\newtheorem{proposition}[theorem]{Proposition}%
\newtheorem{lemma}[theorem]{Lemma}
\newtheorem{corollary}[theorem]{Corollary}

\theoremstyle{thmstyletwo}%
\newtheorem{example}{Example}%
\newtheorem{remark}{Remark}%

\theoremstyle{thmstylethree}%
\newtheorem{definition}{Definition}%

\begin{document}

\title[Basic Albanese maps of regular Riemannian foliations]{Basic Albanese maps of regular Riemannian foliations}


\author{\fnm{Kinga} \sur{S\l{}owik}}

\author{\fnm{Robert} \sur{Wolak}}

\affil{\orgdiv{Institute of Mathematics, Faculty of Mathematics and Computer Science}, \orgname{Jagiellonian University}, \orgaddress{\street{\L{}ojasiewicza 6}, \city{Cracow}, \postcode{30348}, \country{Poland}}}


\abstract{In the paper we introduce the notion of basic Albanese map which we define for foliated Riemannian manifolds using basic 1-forms. We  relate this mapping to the classical Albanese map for the ambient manifold. The study of general properties is supplemented with the description of several important examples.}



\maketitle

\section{Introduction}\label{sec1}

The Albanese map, also called the Abel--Jacobi map, was primarily introduced to study algebraic varieties, in particular complex surfaces. Although this tool originates from algebraic geometry, the Albanese map was also used in differential geometry, for example by Nagano and Smyth \cite{Nagano} to study minimal submanifolds of tori, as well as by Kotschick \cite{Kot} and Nagy and Vernicos \cite{NV} to study geometrically formal manifolds, i.e. manifolds admitting a Riemannian metric such that the wedge product of differential forms preserves harmonicity with respect to the metric. We investigate the Albanese map in the context of Riemannian foliated manifolds and its foliated counterpart. To this end, we need to consider the first basic cohomology which is rational or integer. The integer basic cohomology itself can be an interesting tool to investigate foliated manifolds, for example in \cite{Kordyukov} the assumption that $H^1(M/\fcal) \cap H^1(M,\zbb)$ forms a~lattice in $H^1(M/\fcal)$ is used to introduce basic Seiberg--Witten invariants for taut Riemannian foliations of codimension $4$.

The paper is organized as follows. In Section \ref{notation} we recall, for the reader's convenience, some information and theorems on Riemannian manifolds, Riemannian foliations, Albanese map, and harmonic maps to flat tori. Moreover, we fix assumptions, definitions and notations used further in the paper. Section \ref{integer_cohomo} is devoted to computing the rank of $H^1(M/\fcal,\zbb)$ - or equivalently computing $\dim_\qbb H^1(M/\fcal,\qbb)$ - in case when $\fcal$ is a~Riemannian foliation. {In our considerations we depend on the Molino structure theorem.} First, we consider the general Riemannian foliation and obtain a splitting of $H^1(M/\fcal,\qbb)$ with the first rational basic cohomology of the lifted foliation as a summand. Thus the next step is an assumption that the foliation is transversely parallelizable. Proposition 3.3, an auxiliary result, says that for a transversely parallelizable foliation F, (the equality of dimensions), where Fb is the basic foliation associated to F, i.e. the foliation by the leaf closures.

In Section \ref{BasicAlb} we introduce the basic Albanese map and prove some of its properties. More precisely, Subsection \ref{constr} is devoted to the construction of the basic Albanese map using given generators of $H^1(M/\fcal,\qbb)$, its relationship with classical Albanese map, and the universal property \ref{UP}.
In Subsection \ref{subm} the case of the basic Albanese map being a~submersion is examined. Then it is a~fiber bundle whose fibers are foliated by leaves of the foliation. We prove Lemma \ref{blMetr} which says that the restriction of a~Riemannian foliation is Riemannian -- more precisely, if the foliation is Riemannian with bundle-like metric $g$, then the restricted metric on fiber is bundle-like with respect to the restricted foliation. Then we provide an example of a~foliation which is transversely parallelizable, but its restriction to a~fiber of basic Albanese map is not. Next we consider a~foliation on a~compact Lie group $G$ induced by its dense compact subgroup $W$ with closure $H$ and show that the basic Albanese torus of $(G,\fcal_W)$ and the Albanese torus of $G/H$ are isomorphic. We then pass to subsection \ref{nil} about some regular foliations on nilmanifolds defined using the construction presented in \cite{trans}. This subsection is completed by working out an example of a~$9$-dimensional compact nilmanifold with a~foliation constructed using the Iwasawa manifold.

Finally, in Section \ref{Appl} we consider singular Riemannian foliations obtained by taking the leaf closures and we show Proposition \ref{sing_fol} which says that if $\dim H^1(M/\fcal) = k$, then the foliation restricted to each stratum of the stratification induced by the foliation has codimension at least $k$ and has no isolated compact leaves if $k\neq 0$. 


\section{Preliminaries}\label{notation}

In the paper we assume that all manifolds are smooth, connected and without boundary. Unless stated otherwise, we assume that foliations are regular. For such a~manifold $M$, the Lie algebra of global vector fields is denoted $\Gamma(M)$. For a~vector field $X$ and a~tensor field $T$ on $M$,  $\mathcal{L}_X T$ is  the Lie derivative of $T$ along $X$.

For a~foliated manifold $(M,\fcal)$ we define the complex of basic differential forms $\Omega^*(M,\fcal)$ as $$\Omega^k(M,\fcal) = \{ \alpha \in \Omega^k (M) : \forall X \in T\fcal \ i_X \alpha = 0,  i_Xd\alpha = 0 \}.$$ Since $d$ maps basic forms to basic forms, $(\Omega^*(M,\fcal),d)$ forms a~subcomplex of the de Rham complex. Its cohomology is denoted  $H^*(M/\fcal)$  and called basic cohomology \cite{Tond}. In degree 1 the inclusion of $\Omega^1(M,\fcal)$ in $\Omega^1(M)$ induces an injective map $H^1(M/\fcal) \hookrightarrow H^1(M,\rbb)$. We denote $H^1(M/\fcal,\zbb) := H^1(M/\fcal) \cap H^1(M,\zbb)$, where $H^1(M,\zbb)$ denotes de Rham cohomology classes whose integrals over elements of singular homology $H_1(M,\zbb)$ are integers. The elements of $H^1(M/\fcal,\zbb)$ are called integer first basic cohomology classes of $(M,\fcal)$. Similarly we define rational first basic cohomology classes $H^1(M/\fcal,\qbb) := H^1(M/\fcal) \cap H^1(M,\qbb)$. 

For $k>1$ and for a~foliation $\fcal$ the basic cohomology groups $H^k(M/\fcal)$ may be infinite-dimensional. However, if we assume that $\fcal$ is Riemannian, then $H^k(M/\fcal)$ is finite-dimensional for every $k$.

To a~Riemannian foliation $\fcal$ on a~manifold $M$ we can associate its mean curvature form $\kappa_\fcal$ which is an element of $\Omega^1(M)$. By a theorem of Dom{\'{\i}}nguez (\cite{Tond} or Theorem 4.18 in \cite{Dom}), we can modify the bundle-like metric for $(M,\fcal)$ so that $\kappa_\fcal$ is a~basic $1$-form. Then the 1-form $\kappa_\fcal$ is closed and we can define a~basic cohomology class $[\kappa_\fcal] \in H^1(M/\fcal)$, called the Alvarez class.
We say that a~Riemannian foliation $\fcal$ is taut if, for the metric and $\kappa_\fcal$ modified as above, $[\kappa_\fcal] = 0$. The following conditions are equivalent:
\begin{enumerate}[(a)]
    \item $\fcal$ is taut;
    \item we can choose a~bundle-like metric on $\fcal$ such that the leaves of $\fcal$ are minimal submanifolds;
    \item $H^q(M/\fcal) \cong \rbb$, where $q = \codim \fcal$;
    \item Poincar\'{e} duality holds for $H^*(M/\fcal)$.
\end{enumerate}

Denote by $\mathcal{H}^1(M,g)$ the space of $1$-forms on $M$ which are harmonic with respect to a Riemannian metric $g$. Recall that if $M$ is compact and orientable, then we have the $\rbb$-vector space isomorphism $\mathcal{H}^1(M,g) \cong H^1(M,\rbb)$ given by taking cohomology class of harmonic forms. This isomorphism is called the Hodge isomorphism. Provided that the foliation $\fcal$ is Riemannian (with a bundle-like metric $g$) and transversely oriented, the situation can be generalized to basic harmonic forms \cite{Tond,HC}. Namely, under such assumptions denote by $\hcal^1(M/\fcal,g)$ the space of differential forms which are basic harmonic for the foliation $\fcal$. The generalized Hodge isomorphism also holds, giving $\hcal^1(M/\fcal,g) \cong H^1(M/\fcal)$. However, note that basic harmonic forms need not be harmonic for the same metric.

The following proposition will be used.
\begin{proposition}[Corollary $1$ in \cite{Nagano}]\label{HarmTorus}
     If $T$ is a~flat torus and $(M,g)$ is a~compact oriented Riemannian manifold, then a~smooth mapping $f: M \to T$ is harmonic if and only if $f^*$ maps harmonic $1$-forms on $T$ to harmonic $1$-forms on $M$.
\end{proposition}

Now we introduce the definition and properties of an Albanese map \cite{Griff,Nagano}. Let $(M,g)$ be a~compact oriented Riemannian manifold such that $b_1(M) = k$. Fix harmonic forms $\omega_1,...,\omega_k$ which form a~basis of $\mathcal{H}^1(M,g)$. Denote by $\Tilde{M}$ the universal covering of $M$ and consider a~point $x_0 \in \Tilde{M}$. The vectors in $\rbb^k$ of the form $\left[\int_{\sigma} \omega_i \right]_{i=1,...,k}$, where $\sigma \in H_1(M,\zbb)$, form a~lattice in $\rbb^k$, which we denote by $\Lambda$. Then the map $\Tilde{M} \to \rbb^k$ defined by $x \mapsto \left[\int_{x_0}^x \omega_i \right]_{i=1,...,k}$ is $(\pi_1(M),\Lambda)$-equivariant and so induces a~map $A: M\to \rbb^k/\Lambda$, called the Albanese map for $(M,g)$ and forms $\omega_1,...,\omega_k$. Then $Alb(M) := \rbb^k/\Lambda$ considered with flat metric is called an Albanese torus. This Albanese map satisfies the following properties:
\begin{proposition}\label{AlbHarm}
    An Albanese map is harmonic as a~map from $(M,g)$ to $Alb(M)$ with a~flat metric.
\end{proposition}

\begin{proposition}[Universal property of Alabanese map]\label{AlbUP}
    Let $A:M \to Alb(M)$ be an Albanese map for $(M,g)$. Then for every flat torus $T$ and harmonic map $f: M \to T$ there exists a~unique affine map of flat tori $h: Alb(M) \to T$ such that $f = h \circ A$.
\end{proposition}

In particular, it follows from the Proposition \ref{AlbUP} that an Albanese map is unique up to translations. We will therefore be speaking about \textit{the} Albanese map.

\section{Integer  basic cohomology for Riemannian foliations}\label{integer_cohomo}

The section is dedicated to the investigation of the integer first basic cohomology classes of a~foliated Riemannian manifold, in particular those with dense leaves. We use the Molino structure theorem to reduce the problem to the case of TP foliations. However, first we have to relate the basic cohomology of the Riemannian foliated manifold to the basic cohomology of its orthonormal frame bundle. Its foliation is TP, so are the closures of its leaves. Therefore first, we will demonsrate that a TP foliation with dense leaves  of a closed manifold has trivial rational basic 1-cohomology. In the final section we extend the result to general compact Riemannian manifolds with dense leaves. 

\label{gen}Firstly we consider an auxiliary case of a fibre bundle. Let $M \xrightarrow{\pi} B$ be a~fiber bundle with a connected fiber $F$. Consider the inclusion $i: F \hookrightarrow M$. Suppose the spaces $F, M, B$ are compact and connected. Then we have the following subsequence of the long exact sequence for fibrations
$$ \pi_1(F) \xrightarrow{i_*} \pi_1(M) \xrightarrow{\pi_*} \pi_1(B) \to 0.$$
By Hurewicz isomorphism, which is functorial and is an abelianization for first homotopy groups (and abelianization is right exact) we obtain an exact sequence of singular homology
$$ H_1(F,\zbb) \xrightarrow{i_*} H_1(M,\zbb) \xrightarrow{\pi_*} H_1(B,\zbb) \to 0,$$
and moreover, by taking tensor product with $\qbb$ over $\zbb$ (which is also a~right exact functor) we obtain the exact sequence
$$ H_1(F,\qbb) \xrightarrow{i_*} H_1(M,\qbb) \xrightarrow{\pi_*} H_1(B,\qbb) \to 0.$$
Then by dualizing we obtain exact sequences $$0 \to H^1(B,\zbb) \xrightarrow{\pi^*} H^1(M,\zbb) \xrightarrow{i^*} H^1(F,\zbb),$$ $$0 \to H^1(B,\qbb) \xrightarrow{\pi^*} H^1(M,\qbb) \xrightarrow{i^*} H^1(F,\qbb)$$ and -- due to the de Rham theorem -- we can consider the latter as an exact sequence in de Rham cohomology.
Therefore we have a~short exact sequence
\begin{equation}\label{ses1}
    0 \to H^1(B,\qbb) \xrightarrow{\pi^*} H^1(M,\qbb) \xrightarrow{i^*} i^*(H^1(M,\qbb)) \to 0.
\end{equation}
Note that $H^1(M,\qbb) = H^1(M,\zbb) \otimes_\zbb \qbb$. Thus, the exact sequence (\ref{ses1}) splits as an exact sequence of vector spaces over a~field $\qbb$.

Next,  we assume that $M \xrightarrow{\pi} B$ is a~principal $G$-bundle for a~compact and connected Lie group $G$. Then the tangent bundle $TM$ splits as $\mathcal{H} \oplus \mathcal{V}$, where the vertical bundle $\mathcal{V}$ is formed by $\ker \pi_*$ and the horizontal bundle $\mathcal{H}$ is $G$-invariant. It follows that $\pi_*: \mathcal{H} \to TB$ is an isomorphism fiber-wise and any vector field on $B$ lifts uniquely to a~$G$-invariant section of $\mathcal{H}$.

Now suppose that we have foliations $\fcal_G, \fcal_M, \fcal_B$ on $G,M,B$ respectively and the foliations are such that the maps $i, \pi$ are foliate. Moreover, we assume that the lift of $\fcal_B$ is contained in $\fcal_M$ -- as in the orthonormal frame bundle of the normal bundle of the foliation  (we define the foliation on the orthonormal frame bundle as the lift of the initial foliation). 

For basic cohomology, consider the restrictions $$H^1(B/\fcal_B,\qbb) \xrightarrow{\pi_\fcal^*} H^1(M/\fcal_M,\qbb) \xrightarrow{i_\fcal^*} i^*(H^1(M/\fcal_M,\qbb)).$$ We want to ask whether in this case $\ker i_\fcal^* = \im \pi_\fcal^*$. By definition, $\ker i_\fcal^* = \ker i^* \cap H^1(M/\fcal_M,\qbb)$ (the latter taken as de Rham cohomology classes with rational integrals and here $i^*$ is considered as a~map in de Rham cohomology). Therefore we need to check that $\im \pi_\fcal^* = \im \pi^* \cap H^1(M/\fcal_M,\qbb)$. It translates to the condition that if a~class $\omega \in H^1(M/\fcal_M,\qbb)$ satisfies $\pi^*\alpha = \omega$ for some $\alpha \in H^1(B,\qbb)$ then $\alpha \in H^1(B/\fcal_B,\qbb)$. However, by the assumption that the lift of $\fcal_B$ via $\pi$ is contained in $\fcal_M$, a~vector field $X \in \Gamma(\fcal_B)$ is lifted to a~vector field $\tilde{X} \in \Gamma(\fcal_M)$ and thus we have $0 = \pi^*\alpha(\tilde{X}) = \alpha(\pi_*X) = \alpha(X)$. Since it works for any $X \in \Gamma(\fcal_B)$, we have that $\alpha$ is basic.

As a corollary we obtain that over $\qbb$ we have $$\dim_\qbb H^1(M/\fcal_M,\qbb) = \dim_\qbb H^1(B/\fcal_B,\qbb) + \dim_\qbb W$$ with $W \subset H^1(F/\fcal_F,\qbb)$.

Now we return to the case  described by  the Molino Structure Theorem. First,  we assume that $\fcal$ is transversely orientable, since then as $\widehat{M}$ we can take an $SO(q)$-bundle. By the considerations above, it means we have $H^1(\widehat{M}/\widehat{\fcal},\qbb) = H^1(M/\fcal,\qbb) \oplus W$, where $W$ is a~submodule of $H^1(SO(q),\qbb)$. The latter considerations depend on $q = \codim \fcal$. 

If $q \neq 2$ then $W \subset H^1(SO(q),\qbb) = 0$ and therefore $\dim H^1(\widehat{M},\widehat{\fcal},\qbb)$ equals $\dim H^1(M,\fcal,\qbb)$.

If $q = 2$ then $\widehat{M}$ is a~$SO(2) = \sbb^1$-bundle and we use the Gysin sequence
$$ 0 \to H^1(M) \xrightarrow{\pi^*} H^1(\widehat{M}) \xrightarrow{\pi_*} H^0(M) \xrightarrow{e \wedge \cdot} H^2(M), $$
where $e$ is the Euler class of the circle bundle and $\pi_*$ is integration over fibers. We obtain the following
\begin{itemize}
    \item for $e = 0$ the Gysin sequence induces a~short exact sequence $0 \to H^1(M) \xrightarrow{\pi^*} H^1(\widehat{M}) \xrightarrow{\pi_*} H^0(M) \to 0$ and since $M$ is connected, we have $H^0(M) \cong \rbb \cong H^1(\sbb^1)$ so in particular $i^*(H^1(M)) = H^1(\sbb^1)$. By the fact that maps $i, \pi$ are foliate it follows that $H^1(\widehat{M}/\widehat{\fcal},\qbb) \cong H^1(M/\fcal,\qbb) \oplus H^1(\sbb^1,\qbb)$ and therefore $\dim H^1(M/\fcal,\qbb) = \dim H^1(\widehat{M}/\widehat{\fcal},\qbb) - 1$;
    \item for $e \neq 0$ we have $0 = \ker (e \wedge \cdot) = \im \pi_*$ and thus $\pi^*$ is an isomorphism. Since both $\pi^*$ and the splitting map are foliate, this induces an isomorphism $H^1(M/\fcal,\qbb) \cong H^1(\widehat{M}/\widehat{\fcal},\qbb)$ and therefore the equality of dimensions.
\end{itemize}

Now consider the case when $\fcal$ is not transversely orientable. Let $\mathscr{M}$ be the orientation covering of foliated manifold $(M,\fcal)$ and let $\mathscr{F}$ be the lift of $\fcal$ to $\mathscr{M}$. We can act on $\mathscr{M}$ by $\zbb_2$ the 2-element group which corresponds to an involution. Denote the involution by $\sigma$. Then $q: \mathscr{M} \to M$ is a~twofold covering of $M$, $\mathscr{M}$ is compact and $\mathscr{F}$ is transversely orientable.

The map $\pi^*: H^1(M) \to H^1(\mathscr{M})$ is a~monomorphism which preserves basic forms and whose image is $H^1(\mathscr{M})^{\zbb_2}$. A transfer homomorphism $\tau: H^1(\mathscr{M}) \to H^1(M)$ given by $\tau(\omega) = (i^*)^{-1}(\omega + \sigma^* \omega)$ satisfies $\tau \pi^* = 2id_{H^1(M)}$. Since $\sigma$ is foliate, $\tau \pi^*$ maps $H^1(\mathscr{M}/\mathscr{F},\qbb)$ to $H^1(M/\fcal,\qbb)$.

We have a~short exact sequence
$$ 0 \to H^1(M/\fcal,\qbb) \xrightarrow{\pi^*} H^1(\mathscr{M}/\mathscr{F},\qbb) \to H^1(\mathscr{M}/\mathscr{F},\qbb)^{\zbb_2} \to 0$$
which splits with a~splitting map $\frac{1}{2}\tau$. Therefore we have $$H^1(\mathscr{M}/\mathscr{F},\qbb) = H^1(M/\fcal,\qbb) \oplus H^1(\mathscr{M}/\mathscr{F},\qbb)^{\zbb_2}$$ from which we obtain $$\dim_\qbb H^1(M/\fcal,\qbb) = \dim_\qbb H^1(\mathscr{M}/\mathscr{F},\qbb) - \dim_\qbb H^1(\mathscr{M}/\mathscr{F},\qbb)^{\zbb_2}.$$

\begin{remark}
    It can be shown that Tischler's theorem \cite{Tisch} and its generalization to several non-singular forms \cite{TischGen} can be generalized to foliated manifolds. Namely, if $\omega_1,...,\omega_k$ are basic closed non-singular forms on foliated manifold $(M,\fcal)$, then they induce a~fiber bundle $M\to \tbb^k$ which is moreover foliate (i.e. maps every leaf of $\fcal$ to a~point). The proof follows as in \cite{TischGen} and such constructed submersion is foliate.

    Indeed, each element of $H^1(M/\fcal,\zbb) \subset H^1(M,\zbb)$ corresponds to a~homotopy class of functions $M \to \sbb^1$ with a fixed basepoint. If $f$ is a~representative of a homotopy class $a$ and $\theta$ is a~generator of $H^1(\sbb^1,\zbb)$ (given by rotation coordinate), then $f^*\theta$ is the initial element of $H^1(M/\fcal,\zbb)$. We will prove that if $f^*\theta \in H^1(M,\zbb)$ is basic then $f$ can be choosen to be foliate.

    Suppose otherwise, i.e. in the homotopy class $a$ there is no foliate $f$. It means that any $f:M\to \sbb^1$ is not constant along some leaf, so there is $X\in \Gamma(\fcal)$ such that $Xf \neq 0$, thus $df_x(X_x) \neq 0$ for some $x\in M$, in particular $df_x(X_x) = c \partial_\theta$. Then we have $f^*\theta(X_x) = \theta(df(X_x)) = \theta(c \partial_\theta) = c \neq 0$, so $f^*\theta$ is not basic.
\end{remark}

\subsection{The case of transversely parallelizable foliations}
In this subsection we assume that $M$ is a~compact connected manifold and $\mathcal{F}$ is a~transversely paralellizable foliation on $M$ of codimension $q$.

\begin{lemma}\label{den1}
    Suppose that the leaves of the transversely parallelizable foliation $\fcal$ are dense. Let $\alpha$ be a~closed form in $\Omega^1(M,\fcal)$. Then
    \begin{enumerate}[(a)]
        \item  $[\alpha] = 0 \in H^1(M/\fcal,\rbb)$ implies that $\alpha = 0$;
        \item $\alpha$ is not exact (so $[\alpha] \neq 0$) if and only if $\alpha$ is non-singular.
    \end{enumerate}
\end{lemma}
\begin{proof}
    Assume $[\alpha] = 0$. It means that $\alpha = df$ for some basic function $f$. As basic functions for dense foliations are constant,  $\alpha =0$. This proves $(a)$ 


    Now suppose $[\alpha] \neq 0$. Let $Y$ be a foliate vector field. Then  the function $\alpha(Y):M \to \rbb$ is basic for $\fcal$ and therefore constant. 
    The assumption $[\alpha]\neq 0$ implies that there exists $x\in M$ such that $\alpha_x \neq 0$. Since $\fcal$ is transversely parallelizable, there exists a non-zero foliate vector field $Y$ such that  $\alpha_x(Y_x) \neq 0$. Together with density of the leaves it means that the function $\alpha(Y)$ is constant and non-zero, hence $\alpha$ is non-singular. 
The other implication is covered by (a). 
\end{proof}

\begin{corollary}\label{dense}
     If the transversely parallelizable foliation $\fcal$ has dense leaves, then $H^1(M/\fcal,\qbb) = 0$.
\end{corollary}
\begin{proof}
   Suppose the contrary. Let $\alpha\in \Omega^1(M,\fcal)$ be a closed form such that $0 \neq [\alpha]_\fcal \in H^1(M/\fcal,\qbb)$. As the fundamental group of $M$ is finitely generated  some integer multiple of $\alpha$  defines an integer 1-cohomology class. So we can assume that $0 \neq [\alpha]_\fcal \in H^1(M/\fcal,\zbb)$. Then we integrate $\alpha$ over elements of $H_1(M,\zbb) \cong \pi_1(M)/[\pi_1(M),\pi_1(M)]$. Since the integrals are integers, they give a~discrete subgroup of $\rbb$ so a~lattice $\Lambda$ in $\rbb$. The map induced by integration $\Tilde{M} \to \rbb$ is $(\pi_1(M),\Lambda)$-equivariant so we obtain a map $p: M \to \tbb^1 = \sbb^1$.
 As $[\alpha]$ is non-zero, $\alpha$ is non-singular and therefore the map $p$ is a surjective submersion.
   
      Since $p$ is a surjective submersion, its fibers form a~codimension $1$ foliation on $M$ of  closed leaves. Moreover, every leaf of this foliation contains a~closure of a leaf of $\fcal$, since integrating $\alpha$ along a~leaf gives $0$. However, since leaves of $\fcal$ are dense, the fiber of $p$ equals $M$, which contradicts surjectivity of $p$.
\end{proof}

\begin{remark}
    A transversely parallelizable foliation with dense leaves may admit non-rational basic forms. For example, consider the torus $\tbb^2$ with $\fcal$ the Kronecker foliation (i.e. a foliation obtained by suspending an irrational rotation of $\sbb^1$). Then the leaves are dense, but $H^1(\tbb^2/\fcal) = \rbb$.
\end{remark}

We have the following proposition.
\begin{proposition}\label{TPrank} Let $(M,\fcal)$ be a~foliated manifold with such that $M$ is compact connected and $\fcal$ is transversally parallelizable and of codimension $q$. Let $\fcal_b$ be the basic foliation associated to $\fcal$. Then $H^1(M/\fcal,\qbb) = H^1(M/\fcal_b,\qbb)$.
\end{proposition}
\begin{proof}
    The inclusion $i: \Omega^*(M,\fcal_b) \hookrightarrow \Omega^*(M,\fcal)$ induces a map $i^*: H^1(M/\fcal_b) \rightarrow H^1(M/\fcal)$. The induced map in cohomology is a monomorphism
    (indeed, if $\omega \in \Omega^1(M,\fcal_b)$ satisfies $\omega = df$ for $f\in \Omega^0(M,\fcal) = \Omega^0(M,\fcal_b)$, then $\omega$ is exact in $\Omega^1(M,\fcal)$ as well). In particular, we have the monomorphism $H^1(M/\fcal_b,\qbb) \rightarrow H^1(M/\fcal,\qbb)$. 
    It remains to show that $H^1(M/\fcal_b,\qbb) \rightarrow H^1(M/\fcal,\qbb)$ is an epimorphism. We will prove it by contradiction.

    Let $p_b: M\to M/\fcal_b$ be the projection,  $L$ be a fiber of $p_b$,  and  $j_L: L \to M$ be the inclusion. Denote the foliation $\fcal$ restricted to $L$ by $\fcal_L$.

First we will show that if  a closed $1$-form $\alpha$ representing a class $[\alpha] \in H^1(M/\fcal,\qbb)$  is such that its restriction to each fibre $L$ is $0$ then the form $\alpha$ is $ \fcal_b$-basic.

The tangent bundle $M$ admits the splitting $T\fcal \oplus Q_1 \oplus Q_2$ with  $T\fcal \oplus Q_1  = \fcal_b$. Then the cotangent bundle $T^*M$ admits the splitting  $T\fcal^* \oplus Q_1^* \oplus Q_2^*$. Our assuption on the form $\alpha$ implies that is is a section of $ Q_2^*$. It remains to  show that for any $X \in Q_1$ $d\alpha(X,\cdot) =0$.

- if $Y \in \fcal$ then $d\alpha(X,Y) =0$ as the form $\alpha$ is $\fcal$-basic;

Now fix a point $x \in M$.  We can assume that both vectors are values at $x$  of vector fields from the TP. 

- if $Y \in Q_1 \oplus Q_2$  assume that $X$ is represented by a section of $Q_1$ and $Y$ by a vector field of TP. 

We have $d\alpha(X,Y) = X\alpha(Y) - Y\alpha(X) -\alpha([X,Y])$.  Then $\alpha(X)=0$ $\alpha$  is a section of $Q_2^*$,  $X\alpha(Y)=0$ as 
$\alpha(Y)$ is an $\fcal_b$-basic function, and $\alpha([X,Y])=0$ as $[X,Y]$ is a section of $T\fcal \oplus Q_1$ as $Y$ is foliated for $\fcal_b$ as well. Therefore $d\alpha(X,Y) =0 $



If there exists a closed $1$-form $\alpha$ representing a class $[\alpha] \in H^1(M/\fcal,\qbb)\setminus i^*(H^1(M/\fcal_b,\qbb))$, then $[j_L^* \alpha] \in H^1(L/\fcal_L)$ must be  non-zero and rational for some $L$. This in turn contradicts Corollary \ref{dense}, as the foliation $\fcal_L$ is dense on $L$.
\end{proof}

\subsection{The case of  Riemannian foliations}
Let $(M,\fcal )$ be a compact Riemannian foliated manifold. Suppose that the leaves of $\fcal$ are dense. Consider the fiber bundle $\widehat{M} \to W$ as in the Molino Structure Theorem  with the lifted foliation $\widehat{\fcal}$ being transversely parallelizable. We have $H^1(\widehat{M}/\widehat{\fcal}) = H^1(W) \oplus \Tilde{H}$ where $\Tilde{H} \subset H^1(F/\fcal|_F)$ (Similarly for $\qbb$ coefficients).

The maps can be represented in the following diagram
\begin{center}
{\begin{tikzcd}
F \arrow[r, "i"] & \widehat{M} \arrow[r, "p"] \arrow[d, "\pi"] & W \\
                 & M                                 &  
\end{tikzcd}}
\end{center}
where $F$ is a~closure of a~leaf of the lifted foliation $\widehat{\fcal}$.

Let $\omega\in \Omega^1(M,\fcal)$ be a closed basic 1-form such that $[\omega]\in H^1(M/\fcal,\qbb)$. Then $\pi: \widehat{M} \to M$ is a~fiber bundle projection and $\pi^*[\omega] \in H^1(\widehat{M}/\widehat{\fcal},\qbb)$. {Note that the induced map $\pi^*$ on first cohomology is a monomorphism.}

Lemma \ref{den1} and \ref{dense} generalize to the dense Riemannian foliations.
\begin{proposition}
    Let $\fcal$ be a Riemannian foliation with dense leaves. Let $\omega\in \Omega^1(M,\fcal)$ be a closed basic 1-form. Then
    \begin{enumerate}[(a)]
        \item\label{df1} $[\omega] \neq 0$ if and only if $\omega$ is non-singular, moreover $[\omega]=0$ implies $\omega = 0$;
        \item\label{df2} $H^1(M/\fcal,\qbb) = 0$.
    \end{enumerate}
\end{proposition}

\begin{proof}
(\ref{df1}) If $[\omega]$ is non-zero then for every $F$ which is a fiber of $\widehat{M} \to W$, the form $i_F^*\pi^*\omega \in \Omega^1(F/\widehat{\fcal}_F)$ is non-zero and closed. Since $\widehat{\fcal}_F$ is a transversely parallelizable foliation on $F$ with closed leaves, by Lemma \ref{den1} it follows that $\pi_F^* \omega$ is non-singular.
As the fiber $F$ was arbitrary, it follows that the form $\pi^* \omega$ is non-singular. Since $\pi$ is a surjective submersion, $\omega$ is non-singular as well.

The equality $[\omega] = 0$ means that $\omega = df$ for $f\in \Omega^0(M,\fcal)$ so a constant function, so $\omega = 0$ and thus (\ref{df1}) holds.

(\ref{df2}) By the  Molino Theorem  and by Corollary \ref{dense} we have $H^1(F/\widehat{\fcal}|_F,\qbb) = 0$. Since  $H^1(M /{\fcal},\qbb)$ injects into  $H^1(F/\widehat{\fcal}|_F,\qbb)$ it is the end of the proof.
\end{proof}

\section{Basic Albanese map}\label{BasicAlb}

We begin this section by defining the basic Albanese map and deriving some basic properties. Then we consider some topological and geometric implications of the assumption that the basic Albanese map is a~submersion. Finally, we examine the basic Albanese maps for nilmanifolds with a~particular type of a~foliation.

\subsection{Construction and basic properties}\label{constr}

Suppose we have a~Riemannian foliated manifold $(M,\fcal,g)$. Let $[\omega_1],...,[\omega_b]$ be a~basis of $H^1(M,\qbb)$ such that $\omega_1,...,\omega_k$ are generators of $\hcal^1(M/\fcal)$. Denote by $\Tilde{M}$ the universal covering of $M$. Similarly as in the classical case, we define a~map 
\begin{equation*}
    \Tilde{A}_\fcal: \Tilde{M} \to \rbb^k, \ [\gamma] \mapsto \left[ \int_{\gamma} \omega_i \right]_{i=1,...,k}.
\end{equation*}
Then the subgroup $\Lambda_\fcal \subset \rbb^k$ generated by $\left[ \int_{\sigma} \omega_i \right]_{i=1,...,k}$, where $\sigma$ are elements of the free part of $H_1(M,\zbb)$, is a~lattice in $\rbb^k$. Indeed, it is an Abelian subgroup of $\qbb^k$ generated by $k$  elements, so it is a~discrete subgroup of $\rbb^k$. \label{MissingPart} 
The map $\Tilde{A}_\fcal$ is $\left( \pi_1(M), \Lambda_\fcal \right)$-equivariant so it factorizes to a~map $A_\fcal: M \to \rbb^k/\Lambda_\fcal$.
\begin{definition}
    The map $A_\fcal$ defined above is called the \textit{basic Albanese} map for $\omega_1,...,\omega_k$ and the torus $Alb(M,\fcal) := \rbb^k/\Lambda_\fcal$ with a flat metric is called the \textit{basic Albanese torus} for $\omega_1,...,\omega_k$.
\end{definition}

The following properties of this basic Albanese map generalize some properties of classical Albanese map and are easy to prove.
\begin{lemma}\label{BasicAlbProperties}
    Let $A_\fcal$ be the basic Albanese map for $\omega_1,...,\omega_k$. Then
    \begin{itemize}
        \item $A_\fcal$ is smooth;
        \item the differential $dA_\fcal : TM \to TAlb(M,\fcal)$ is given explicitly by $dA_\fcal = (\omega_1,...,\omega_k)$, in particular $A_\fcal$ is a~submersion if and only if all forms $\omega_1,...,\omega_k$ are pointwise linearly independent, i.e. the form $\omega_1\wedge ... \wedge \omega_k$ is non-singular;
        \item every leaf of the foliation $\fcal$ (and therefore its closure) is contained in some fiber of $A_\fcal$.
    \end{itemize}
\end{lemma}

Basic Albanese map for given forms may not be harmonic in general. However, this is the case when $\kappa_\fcal$ is basic, since then basic harmonic forms are harmonic (Lemma 2.4 in \cite{HRW}).

\begin{proposition}\label{rmkAlb} Let $(M,g,{\fcal})$ be a compact foliated Riemannian manifold with the basic Alvarez class $\kappa_\fcal$. Let $b = dim H^1(M)$ and $k = dim H^1(M/\fcal) $. Then there exists a linear map $\tilde\pi: \rbb^b \to \rbb^k$ such that $\pi A = {A}_{\fcal}$ where $ A$ and ${A}_{\fcal}$ are the Albanese and the basic Albanese map, respectively. 
\end{proposition}

\begin{proof}
 The basic Albanese map is related to the classical Albanese map in the following way. Complete $[\omega_1],...,[\omega_k]$ to a~basis of $H^1(M,\rbb)$ with forms $\omega_{k+1},...,\omega_b \in \hcal^1(M,g)$. Denote by $\tilde\pi: \rbb^b \to \rbb^k$ the projection onto first $k$ coordinates and by $A_M$ the Albanese map of $M$. The set $\Lambda \subset \rbb^b$ generated by $\left[ \int_{\sigma_j} \omega_i \right]$ for $i,j = 1,...,b$ and $\sigma_j$ being generators of torsionless part of $H_1(M,\zbb)$ is again a~lattice, so $\tilde\pi$ induces a~map $\pi: Alb(M)\to Alb(M,\fcal)$. Then $\pi$ maps $\Lambda$ to $\Lambda_\fcal$ and we have a~commutative diagram
    \begin{center}
        \begin{tikzcd}
        \rbb^b \arrow[r, "\tilde\pi"]                 & \rbb^k \\
        \Tilde{M} \arrow[u, "\Tilde{A}"] \arrow[ru, "\Tilde{A}_\fcal"'] &  
    \end{tikzcd}
    \end{center}
     which after taking quotients descends to
     \begin{center}
         \begin{tikzcd}
        Alb(M) \arrow[r, "\pi"]                 & Alb(M,\fcal) \\
        M \arrow[u, "{A}"] \arrow[ru, "{A}_\fcal"'] &  
    \end{tikzcd}.
     \end{center}
\end{proof}

Moreover, in this case the basic Albanese mapping is harmonic.

\begin{proposition} Let $(M,g,{\fcal})$ be a compact foliated Riemannian manifold with the basic Alvarez class $\kappa_\fcal$.
   Then the map $A_\fcal$ is a~harmonic map.
\end{proposition}
\begin{proof}
    By Proposition \ref{HarmTorus} it is enough to check that $A^*_\fcal$ maps harmonic 1-forms on $Alb(M,\fcal)$ to harmonic 1-forms on $M$.

    By Remark \ref{rmkAlb} we have $A_\fcal^* = A^* \pi^*$. Since $\pi$ is a~harmonic map as an affine map of flat tori and $A$ is a~harmonic map to a~flat torus by Proposition \ref{AlbHarm}, they preserve harmonic $1$-forms, and so does their composition.
\end{proof}

\begin{proposition} Let $(M,g,{\fcal})$ be a compact foliated Riemannian manifold with the basic Alvarez class $\kappa_\fcal$.
    Then the map $A_\fcal^* : H^1(Alb(M,\fcal),\zbb) \to H^1(M/\fcal,\zbb)$ is an isomorphism.
\end{proposition}
\begin{proof}
     The forms $dx_i$, $i = 1,...,k$, are representatives of  generators of $H^1(Alb(M,\fcal))$ and for any $X \in\Gamma(M)$ we have
\begin{equation*}
    (A_\fcal^*dx_i)(X) = dx_i(dA_\fcal \, X) = dx_i((\omega_1(X),...,\omega_k(X))) = \omega_i(X). \qedhere
\end{equation*}
\end{proof}

Similarly as in the case of classical Albanese map, we have the following universal property.
\begin{proposition}\label{UP}
        Let assume that the mean curvature 1-form $\kappa_\fcal$ of the foliation $\fcal$   is basic. Then for any flat torus $T$ and any harmonic map $f: M \to T$ mapping any leaf of $\fcal$ to some point there exists a~unique affine map of flat tori $h: Alb(M,\fcal) \to T$ such that $f = h \circ A_\fcal$.
\end{proposition}
\begin{proof}
    The proof follows similarly as the proof of Proposition 1 in \cite{Nagano}. Consider the map $f^*: H^1(T) \to H^1(M)$ induced by $f$. For any $\omega \in H^1(T)$ and $X \in \Gamma(\fcal)$ we have $(f^* \omega)(X) = \omega(f_* X) = 0$ because $f_* X=0$ as $f$ maps leaves of $\fcal$ to points. Therefore $f^* \omega \in H^1(M/\fcal)$.

    The map $(A_\fcal)^{-1} \circ f^*: H^1(T) \to H^1(Alb(M,\fcal))$ then induces the unique affine map $h: T \to Alb(M,\fcal)$ satisfying $f = h \circ A_\fcal$, similarly as in \cite{Nagano}.
\end{proof}

Therefore we can talk about \textit{the basic Albanese map of} $(M,\fcal)$.

\begin{corollary}
    Let $(M,\fcal)$ be a~foliated manifold with transversely parallelizable $\fcal$. Let $W$ be its basic manifold. Then $Alb(W)$ and $Alb(M,\fcal)$ are affine diffeomorphic  via harmonic diffeomorphisms.
\end{corollary}

It is a simple consequence of Proposition \ref{TPrank}.

\subsection{Albanese map is a submersion}\label{subm}
In this subsection we always assume that a basic Albanese map is a~submersion.

Recall a statement from Lemma \ref{BasicAlbProperties} that the basic Albanese map for rational basic harmonic forms $\omega_1,...,\omega_k$ is a~submersion if and only if $\omega_1\wedge ...\wedge \omega_k$ is non-singular. This property permits us to present   the following examples of foliations with submersive basic Albanese map.

\begin{example}
    Let $(M,g)$ be a~Riemannian  manifold with $g$ being a bundle-like metric for a transversely orientable foliation $\fcal$ of $M$  and formal, i.e. the wedge product of any  two $g$-harmonic forms is a $g$-harmonic. \cite{Kot}. Assume that the mean curvature form $\kappa_\fcal$ is basic. Indeed, by the Hodge theorem for basic forms (Theorem 7.22 in \cite{Tond}) we can represent any basic cohomology class by a~basic harmonic form (basic harmonic with respect to $g$). If the mean curvature form $\kappa_{\fcal}$ is basic, then basic harmonic forms are harmonic. Choose basic harmonic forms $\omega_1,...,\omega_k$ such that $[\omega_1],...,[\omega_k]$ generate $H^1(M/\fcal,\qbb) \subset H^1(M,\qbb)$. Since the forms are non-zero and harmonic, their norms the norms of their harmonic wedge products are non-zero constants \cite{Kot} and so their wedge product is nonsingular.
        
\end{example}

\begin{example}[Transversely geometrically formal foliations \cite{HRW}]
    Recall that a transversely oriented Riemannian foliation is transversely geometrically formal if it admits a bundle-like metric $g$ such that interior product of two basic harmonic forms is basic harmonic. Such foliations are transversely orientable by definition, and therefore basic forms admit Hodge decomposition (Theorem 7.22 in \cite{Tond}). Using this decomposition, we can take basic harmonic representatives of each cohomology class. Therefore, as $\omega_1,...,\omega_k$ we can choose basic harmonic forms, whose norm is also constant by Lemma 3.4 of \cite{HRW}.
\end{example}

\begin{proposition}
    Suppose that the basic Albanese map is a~submersion. Then $\dim H^1(M/\fcal,\qbb) \leqslant \codim \fcal$. Moreover, if there is equality, the basic harmonic forms have constant length and $\kappa_\fcal$ is basic then the leaves of $\fcal$ are minimal.
\end{proposition}

\begin{proof}

First let us remark the following: 

If  $A_\fcal: M \to \tbb^k$ be a~submersion, then the closures of leaves of $\fcal$ are submanifolds of codimension at least $k$. Indeed let $L$ be a~leaf of the foliation $\fcal$. Denote by $F$ the fiber of $A_\fcal$ which contains $L$. Then $\codim_M F = k$ and, since $\overline{L}$ is a submanifold of $F$, $\codim_M \overline{L} \geqslant k > 0$.

The minimality of leaves results from the previous considerations.
\end{proof}

\begin{remark}
    The inverse implication is not true. Indeed, consider a~Riemannian foliation on a~compact manifold $M$ induced by a~submersion to $\sbb^k \neq M$ with $k>1$. For example one can take a~foliation induced by a~Hopf fibration \cite{wilk}. Then the first basic cohomology is isomorphic to $H^1(\sbb^k) = 0$ and the leaves are closed, so not dense in $M$.
\end{remark}

By Ehresmann theorem (Lemma 9.2. in \cite{Michor}), the assumption that $A_\fcal$ is a~submersion implies that $A_\fcal$ is a~fiber bundle with a~fiber $F$.
We can therefore consider the restricted foliation $\fcal|_F$. Let $i: F \to M$ be the inclusion map.
   Even if the fiber $F$ of basic Albanese map is a~minimal submanifold of $M$, it need not be a~totally geodesic submanifold. Therefore there is no obvious relation between geodesics on $M$ and $F$. Moreover, if we assume $\fcal$ is Riemannian, it is natural to ask whether the restricted metric $g_F$ is bundle-like for $\fcal|_F$. Indeed, we have


\begin{lemma}\label{blMetr}
    Let $g$ be a~Riemannian metric on $M$ which is bundle-like for $\fcal$. Then the restricted metric $g_F$ on $F$ is bundle-like for $\fcal|_F$.
\end{lemma}
\begin{proof}
    By \cite{MetrFol} the claim that $g_F$ is bundle-like for $\fcal|_F$ is equivalent to the claim that $g_F(\nabla_X^F V, Y) + g_F(X,\nabla_Y^F V) = 0$ for any vector field $V \in \Gamma(F)$ which is tangent to the leaves of $\fcal|_F$ and any vector fields $X,Y \in \Gamma(F)$ which are orthogonal to leaves of $\fcal|_F$ and $\nabla^F$ is the Levi-Civita connection for $(F,g_F)$.

    Choose $X,Y \in \Gamma(F)$ and extend them to vector fields on $M$ and denote the extensions by the same symbol, then by Gauss formula
    $$ \nabla_X^M Y = \nabla_X^FY + \alpha(X,Y) $$
    where $\alpha: \Gamma(F) \times \Gamma(F) \to \Gamma(F)^\perp$ is the second fundamental form.

    We can extend vector fields on $F$ which are foliate with respect to $\fcal|_F$ to vector fields on $M$ so that they are foliate with respect to $\fcal$ and, when we use local trivialization for fiber bundle and consider a~trivialized neighbourhood $U \times F$, in which we identify initial fiber $F$ with $\sing{x} \times F$, the vector fields are of the form $\sum f_i X_i$ on $U\times F$, where each $X_i$ is a~vector field of $F$ and each $f_i \in C^\infty(U)$ is such that there exist neighbourhoods $x \in V_1 \subset V_2 \subsetneq U$ satisfying $f_i|_{V_1} \equiv 1$ and $f_i|_{U\setminus V_2} = 0$, and the vector fields are extended to $0$ on $M$ without the trivialized neighbourhood.
    Denoting the extended vector fields by the same symbols, we have
    \begin{equation*}
        \begin{split}
            0 &= g(\nabla_X^M V, Y) + g(X, \nabla_Y^M V) \\&= g(\nabla_X^F V + \alpha(X,V), Y) + g(X, \nabla_Y^F V + \alpha(Y,V)) \\&= g(\nabla_X^F V, Y) + g(X, \nabla_Y^F V)
        \end{split}
    \end{equation*}
    which proves the initial claim.
\end{proof}


\begin{example}
{Even when the initial foliation is transversely parallelizable, the foliation restricted to a~fiber of basic Albanese map need not have this property. Let $\fcal$ the Riemannian foliation given by a~Hopf fibration, for example a~foliation on $\sbb^{3}$ induced from submersion $\sbb^{3} \to \sbb^2$. Consider $\sbb^{3} \times \tbb^1$ and a~product foliation $\mathcal{G} = \fcal \times *$, where $*$ denotes the foliation on $\tbb^1$ by singletons. Then the basic Albanese torus of the foliation is $\tbb^1$ and the fiber of basic Albanese map is $\sbb^{3}$ with the induced foliation being $\fcal$.}

    The foliation $\fcal$ is not transversely parallelizable since $\sbb^2$ is not parallelizable. However, as $\sbb^2 \times \tbb^1$ is parallelizable as a~$3$-dimensional closed manifold, the foliation $\mathcal{G}$ is transversely parallelizable.
\end{example}

\begin{example}\label{GrpAction}
    Let $G$ be a~compact connected Lie group. Let $H$ be a~closed and connected subgroup of $G$ (so $H$ and $G/H$ are compact manifolds) and suppose $W$ is a~dense connected subgroup of $H$. Denote respectively by $\fcal_W, \fcal_H$ the foliation of $G$ by cosets of $W, H$. The leaves of $ \fcal_H$  are compact. 

    The foliations $\fcal_W, \fcal_H$ are homogeneous and it can be shown that the basic foliation of $\fcal_W$ and the  foliation $\fcal_H$ are equal, therefore $G/H$ is the~basic manifold for $\fcal_W$. Applying the discussion in Subsection \ref{gen} to the principal $H$-bundle $G \to G/H$ we obtain an isomorphism $H^1(G/\fcal_W,\zbb) \cong H^1(G/H,\zbb)$ (since $H^1(H,\fcal_W|_H) = 0$ as $\fcal_W|_H$ is transversely parallelizable on $H$ and has dense leaves) and therefore $Alb(G,\fcal)$ is diffeomorphic to $Alb(G/H)$.
\end{example}

\subsection{Application to foliations on nilmanifolds}\label{nil}

In this subsection, we apply the basic Albanese map to investigate foliations on nilmanifolds which are induced by surjective submersions to nilpotent simply connected Lie groups.

We recall the more precise construction from \cite{CW}. Let $N$ be a~nilpotent, simply connected Lie group. Let $\Gamma \subset N$ be a~torsion-free finitely generated subgroup which contains a subgroup $\Gamma_0$ that is a~cocompact group for $N$. Then by a~theorem of Malcev \cite{malcev} there exists a~nilpotent, simply connected Lie group $N_1$ which contains $\Gamma_1$ as a~cocompact group and there is a~Lie group homomorphism $p: N_1 \to N$ such that
\begin{enumerate}[(a)]
    \item $p|_{\Gamma_1}: \Gamma_1 \to \Gamma_0$ is an isomorphism;
    \item\label{latter} $p$ is a~surjective submersion with connected fibers.
\end{enumerate}
By the latter property (\ref{latter}), the fibers of $p$ define a~foliation $\fcal_1$ on $N_1$. Since the foliation is $\Gamma_1$-invariant, it projects to a~foliation $\fcal$ on $M := N_1/\Gamma_1$. Thus we have the following diagram
\begin{center}
    \begin{tikzcd}
N_1 \arrow[r, "p"] \arrow[d, "\pi_{\Gamma_1}"'] & N \arrow[d, "\pi_{\Gamma_0}"] \\
M                                           & N/\Gamma_0                   
.\end{tikzcd}
\end{center}
The foliation $\fcal$ inherits transverse geometric structure from the foliation $\fcal_1$. In particular, if $\fcal_1$ is Riemannian then $\fcal$ is also Riemannian.
\medskip

Let $\Omega^*(N;\Gamma)$ denote the $\Gamma$-invariant forms on $N$ and $\Omega^*(N_1/\fcal_1;\Gamma_1)$ denote the $\Gamma_1$-invariant $\fcal_1$-basic forms on $N_1$. We have the following isomorphisms
\begin{equation}\label{isos}
    \Omega^*(M/\fcal) \xrightarrow[\cong]{\pi^*_{\Gamma_1}} \Omega^*(N_1/\fcal_1;\Gamma_1) \xleftarrow[\cong]{p^*} \Omega^*(N;\Gamma).
\end{equation} 
Moreover, $\pi_\Gamma^*$ induces an isomorphism in basic cohomology and we will use the fact that for $\omega \in H^1(M/\fcal)$ and $\sigma \in H_1(M,\zbb)$ we have
$$ \int_\sigma \omega = \int_{\Tilde{\sigma}} \pi_\Gamma^* \omega $$
where $\Tilde{\sigma}$ is a~path in $N_1$ which lifts $\sigma$.

    We have a~chain of inclusions $\Omega^*(N;N) \hookrightarrow \Omega^*(N,\Gamma) \hookrightarrow \Omega^*(N,\Gamma_0)$ and by the Nomizu theorem \cite{Nomizu} the composition induces an isomorphism in cohomology $H^*(N;N) \cong H^*(N/\Gamma_0)$. Also note that $\Omega^*(N;N) \cong A^*(\mathfrak{n})$ where $\mathfrak{n}$ is the Lie algebra of $N$ and $A^*(\mathfrak{n})$ is the Chevalley--Eilenberg complex of $\mathfrak{n}$, where $\mathfrak{n}$ is the Lie algebra associated to $N$. Therefore we have $H^*(\mathfrak{n}) \cong H^*(N/\Gamma_0)$.

One may ask about the relationship between the Albanese torus of $N/\Gamma_0$ and the basic Albanese torus of $(M,\fcal)$. To give such a~relationship, we first describe the natural   action of $N$ on $\Omega^*(M/\fcal)$. 

    Choose $\omega \in \Omega^*(M/\fcal)$, then by isomorphisms (\ref{isos}) there exists $\alpha \in \Omega^*(N;\Gamma)$ such that $\pi_\Gamma^* \omega = p^* \alpha$, or equivalently $\omega = (\pi_\Gamma^*)^{-1} p^* \alpha$. We will define an action of $N$ on $\Omega^*(M/\fcal)$ as follows. For $\xi \in N$ and $\omega \in \Omega^*(M/\fcal)$ such that $\omega = (\pi_\Gamma^*)^{-1} p^* \alpha$ we define $\xi.\omega := (\pi_\Gamma^*)^{-1} p^* L_\xi^*\alpha$. Note that $\omega$ is $N$-invariant, i.e. $N_\omega = N$, if and only if $\alpha \in \Omega^*(N;N)$.

    The inclusion map $\Omega^*(N;N) \hookrightarrow \Omega^*(N;\Gamma)$ induces a monomorphism $j^*: H^*(N/\Gamma_0) \hookrightarrow H^*(M/\fcal)$. By isomorphisms (\ref{isos}), for any $[\omega] \in H^*(M/\fcal)$ there is a~closed $\alpha \in \Omega^*(N,\Gamma)$  such that $\omega = (\pi_\Gamma^*)^{-1} p^* \alpha$. Moreover, the map $j^*: H^*(N/\Gamma_0) \hookrightarrow H^*(M/\fcal)$ is induced by composition $$\Omega^*(N;N) \hookrightarrow \Omega^*(N;\Gamma) \xrightarrow[\cong]{(\pi_\Gamma^*)^{-1} p^*} \Omega^*(M/\fcal)$$ and is well defined because those maps commute with de Rham differentials. Therefore, $\im j^*$ consists precisely of elements of the form $(\pi_\Gamma^*)^{-1} p^* \alpha$ for $\alpha \in \Omega^*(N;N)$.

In general, the monomorphism $H^1(N/\Gamma_0) \to H^1(M/\fcal) \subset  H^1(M)$ need not preserve rational classes, as the following example shows.

\begin{example}
    Consider $\rbb$ with addition (in particular it is an Abelian Lie group so a~nilpotent Lie group) and $\Gamma = \left< 1, \eta \right>$ where $\eta \in \rbb\setminus\mathbb{Q}$. Such $\Gamma$ is a~torsionfree subgroup of $\rbb$ and contains $\Gamma_0$ generated by $1$ as a~cocompact group for $\rbb$.

    Then for $\rbb^2$ with addition and $\Gamma_1 = \left< (1,0), (\eta, 1) \right>$ consider the map $p: \rbb^2 \to \rbb$ given by projection on the first coordinate. The map $p$ is obviously a~surjective homomorphism of Lie groups which is also a~submersion and has connected fibers (isomorphic to $\rbb$). Moreover, $p$ maps $\Gamma_1$ isomorphically to $\Gamma$. The fibers of $p$ induce a~dense linear foliation on the torus $T := \rbb^2/\Gamma_1$, denote the foliation by $\fcal$ and thus $H^1(T/\fcal)$ contains only irrational classes. Therefore the monomorphism $H^1(\rbb/\Gamma_0) \hookrightarrow H^1(T/\fcal)$ maps integer classes to irrational classes. 
\end{example}

\begin{example}
    This example comes from \cite{CW}. Calculations in this section were done using SageMath. More precisely, we consider the Iwasawa manifold $\Gamma_0 \backslash N$, where $N$ is the complex Lie group with $\dim_\cbb N = 3$ whose elements are given by matrices of the form $$\left(\begin{array}{rrr}
            1 & z_1 & z_3 \\
            0 & 1 & z_2 \\
            0 & 0 & 1
        \end{array}\right)$$ and $\Gamma_0$ is a~subgroup of $N$ consisting of matrices such that $z_1,z_2,z_3 \in \zbb[i]$.
    It can be shown that $N$ contains a~subgroup $\Gamma$ which contains $\Gamma_0$ and consists of matrices of the form
        $$\left(\begin{array}{rrr}
            1 & x_1 + i(y_1 + sy_1') & x_3 + sx_3' + i(y_3 + sy_3') \\
            0 & 1 & x_2 + iy_2 \\
            0 & 0 & 1
        \end{array}\right)$$
    where $s\in \rbb \backslash \qbb$ and $x_i, x_i', y_i, y_i' \in \zbb$.

    Consider $N_1 := \rbb^9$ with group action given by
    \begin{equation*}
            \begin{split}
                (a_1,...,a_9)\star (x_1,...,x_9) =(&a_i + x_i, 
       a_6 + x_6 + a_1x_4 - a_2x_5,\\& a_7 + x_7- a_3x_5, a_8 + x_8 + a_1x_5 + a_2x_4, \\&a_9 + x_9 + a_3x_4).
            \end{split}
    \end{equation*}
    Denote by $\Gamma_1$ the uniform subgroup of $N_1$ consisting of elements with integer terms. Then the map $u: N_1 \to N$ defined as $$(x_{1}, ..., x_{9}) \longmapsto (x_{1} + i(x_{2} + sx_3), x_{4} + i x_{5}, x_6 + s x_{7} + i (x_{8} + sx_9))$$ is a~submersion of Lie groups which maps $\Gamma_1$ isomorphically to $\Gamma$. This submersion has connected fibers which form a~foliation that descends to a~foliation $\fcal$ of $\Gamma_1\backslash N_1$. The foliation is transversely symplectic, transversely holomorphic and is of complex codimension $3$, so of real dimension $3$.

    The leaves of this foliation can be parametrized as follows. If we denote by $L$ the leaf given by the fiber of $u$ over a~point $(z^0_1,z^0_2,z^0_3)$ with $z^0_j = x^0_j + iy^0_j$, then we have a~map $\rbb^3 \to L \subset N_1$ given by $$\left(r_{1}, r_{2}, r_{3}\right) \longmapsto \left(x^0_1, -r_{1} s + y^0_1, r_{1}, x^0_2, y^0_2, -r_{2} s + x_3^0, r_{2}, -r_{3} s + y_3^0, r_{3}\right).$$ From this we can find vector fields on $\Gamma_1 \backslash N_1$ which are tangent to the leaves and they are generated by $$v_1' = -s {\partial_{x_{2}} } +{\partial_{x_{3}} }, v_2' = -s {\partial_{x_{6}} } +{\partial_{x_{7}} }, v_3'= -s {\partial_{x_{8}} } +{\partial_{x_{9}} }.$$

    On $\Gamma_1\backslash N_1$ we have the following generating left-invariant vector fields
        \begin{equation*}
        \begin{array}{c}
             X_i = \partial_{x_i}, i = 1,2,3 \\
             X_4 = {\partial_{x_{4}} } + x_{1} {\partial_{x_{6}} } + x_{2} {\partial_{x_{8}} } + x_{3} {\partial_{x_{9}} } \\
             X_5 = {\partial_{x_{5}} } -x_{2} {\partial_{x_{6}} } -x_{3} {\partial_{x_{7}} } + x_{1} {\partial_{x_{8}} } \\
             X_j = \partial_{x_j}, j = 6,...,9
        \end{array}
    \end{equation*}    
    {(observe that $\partial_{x_6} = [X_1,X_4], \partial_{x_7} = -[X_3,X_5], \partial_{x_8} = [X_2,X_4]$ and $\partial_{x_9} = [X_3,X_4]$)}
    and their corresponding dual left-invariant $1$-forms
    \begin{equation*}
        \begin{array}{c}
             w_i = dx_i, i = 1,...,5 \\
             w_6 = -x_1 dx_4 + x_2 dx_5 + dx_6 \\
             w_7 = x_3 dx_5 + dx_7 \\
             w_8 = -x_2 dx_4 - x_1 dx_5 + dx_8\\
             w_9 = -x_3 dx_4 + dx_9
        \end{array}
    \end{equation*}
    among which only $dx_1, dx_2, dx_3, dx_4, dx_5$ are closed and not exact. Therefore they generate $H^1(\Gamma_1\backslash N_1)$. 

    We define a~Riemannian metric $g$ on $\Gamma_1 \backslash N_1$ so that the vector fields $X_1,...,X_9$ are orthonormal. With respect to this metric, the vector fields $v_i'$ have norm $\sqrt{s^2 + 1}$ so we take normalizations $v_i:= \frac{1}{\sqrt{s^2 + 1}}v_i'$. {By calculations the vector fields $v_1,v_2,v_3$ form an orthonormal frame of leaves of $\fcal$.} The vector fields $$X_i' := X_i - g(X_i,v_1)v_1 - g(X_i,v_2)v_2 - g(X_i,v_3)v_3$$ are orthogonal to leaves of $\fcal$. It can be computed that for any $i,j$ we have
    \begin{itemize}
        \item $g(X_i',X_j')$ is a~constant
        \item for any $k =1,2,3$ the Lie derivative is given by $$\mathcal{L}_{v_k}g = g([v_k,\bullet],\circ) + g(v_k,[\bullet,\circ]) = 0$$
    \end{itemize}
    therefore the metric $g$ is bundle-like for the foliation $\fcal$.
    
    By direct calculations it can be shown that for the Levi-Civita connection $\nabla$ and for $k=1,2,3$ we have $\nabla_{v_k}v_k = 0$ and thus $\kappa$ the mean curvature form for $\fcal$ satisfies $\kappa^\sharp = 0$ meaning $\kappa = 0$. In particular, basic harmonic forms are harmonic and the foliation is taut.
    
    {Then} $\Omega^1(\Gamma_1\backslash N_1,\fcal_1)$ is generated by $$ \omega_1 = dx_2 + sdx_3,\ \omega_2 = dx_6 + sdx_7,\ \omega_3 = dx_8 + sdx_9,\ dx_1,\ dx_4,\ dx_5,\ $$
    among them only $\omega_1, dx_1, dx_4, dx_5$ are closed and not exact and so they are representatives of $H^1(\Gamma_1\backslash N_1/\fcal_1)$. Among the closed exact basic forms only $\omega_1$ is an irrational form. On the other hand, the forms $dx_1,dx_4,dx_5$ are basic harmonic. Indeed, we compute $\delta_\fcal = \pm \Bar{\ast} d \bar{\ast}$ and $ d(\bar{\ast}dx_i) = d(\pm \ast(dx_i \wedge \chi))= 0$ for $i=1,4,5$. Therefore we have dimension $3$ basic Albanese torus. Moreover, the forms $dx_1,dx_4,dx_5$ are non-singular so the map is a~submersion.
    
    The basic Albanese map $A_\fcal$ can be given by projection to $(x_1,x_4,x_5)$ modulo $\zbb^3$, so $$A_\fcal^{-1}(P_1,P_4,P_5) = (P_1,x_2,x_3,P_4,P_5,x_6,...,x_9)$$ and then $(0,a_2,a_3,0,0,a_6,...,a_9) \star (P_1,x_2,x_3,P_4,P_5,x_6,...,x_9)$ is equal to
    \begin{equation*}
        \begin{split}
            (&P_1,a_2 + x_2, a_3 + x_3,P_4,P_5, a_6 + x_6 - a_2P_5, a_7 + x_7 - a_3 P_5, \\&a_8 + x_8 + a_2P_4, a_9 + x_9 + a_3P_4)
        \end{split}
    \end{equation*}which is again in the fiber.

    Note that each fiber is a~torus -- indeed, the fiber over $(0,0,0)$ is a~torus and the left action by $(P_1,0,0,P_4,P_5,0,0,0,0)$ on the fiber induces a~diffeomorphism between $A_\fcal^{-1}((0,0,0))$ and $A_\fcal^{-1}((P_1,P_2,P_3))$. This diffeomorphism is also right (but not left) invariant and foliate for $\fcal$ restricted to fibers. The restricted foliation is parametrized by $(P_1, -r_{1} s + y^0_1, r_{1}, P_2, P_3, -r_{2} s + x_3^0, r_{2}, -r_{3} s + y_3^0, r_{3})$, and since $s$ is irrational the foliation is dense (because e.g. it is dense on the torus being a~fiber over $(0,0,0)$).

    On the Iwasawa manifold $\Gamma \backslash N$ the first real cohomology is represented by forms $dx_1, dy_1, dx_2,  dy_2$ and $p^*$ maps them to $dx_1, \omega_1, dx_4, dx_5 $ respectively. In this case we have 3 integer forms and one irrational form in the image of $H^1(\Gamma \backslash N,\zbb) \to H^1(\Gamma_1 \backslash N_1)$.
\end{example}

\section{Some applications to the closure of leaves of a  Riemannian foliation}\label{Appl}

In this section we use the basic Albanese map to obtain some information on  the dimension the   closures of leaves of a Riemannian foliation, cf. \cite{Molino}.

Let $(M,g,{\mathcal F})$ be a~compact foliated Riemannian manifold of dimension $n$ with a~regular Riemannian foliation $\fcal$ of codimension $q$ for which $g$ is bundle-like. Then if $\alpha$ is a 1-basic form, then its dual with respect to $g$ $\alpha^{\sharp}$ is an $\mathcal F$-orthogonal foliate vector field. The points of leaves whose closure have the same dimension are submanifolds of $M$ and form a stratification of the manifold $M$. Moreover, global foliate vector fields are {tangent} to these strata of $(M,g,{\mathcal F})$.  Let $X$ be a~foliate vector field of $\fcal$ and consider the associated global flow $\phi_X: M \times \rbb \to M$. Then $\phi_X^t: M \to M$ is a~diffeomorphism of $M$.  It maps diffeomorphically a~leaf $L$ to another leaf $L'$, and therefore their  closures. Thus  each $\phi_X(t)$ preserves the stratification, hence the vector field must be tangent to the strata. 



Therefore, we have the following

\begin{proposition}\label{sing_fol} Let $(M,g,{\mathcal F})$ be a~Riemannian foliated manifold such that $\dim H^1(M/{\mathcal F})=k$ and the basic Albanese map is a~submersion. Then the foliation $\mathcal F$ restricted to each stratum has codimension at least $k$. In particular, if $k >0$, the foliation $\mathcal F$ has no isolated compact leaves.
\end{proposition}

\begin{proof}
    By the assumption that the basic Albanese map is a~submersion, the 1-forms which are representatives of non-zero classes of $H^1(M,\fcal)$ cannot vanish at any point.   It follows that we have basic $1$-forms $\alpha_1,...,\alpha_k$ which are orthogonal 
    with respect to the metric $\Tilde{g}$ induced by $g$. Then $\alpha_1^\sharp, ..., \alpha_k^\sharp$ are non-vanishing and orthogonal with respect to $g$. Moreover, they are also tangent to strata  and orthogonal to leaves of $\fcal$, therefore the codimension of $\fcal$ on each stratum is at least $k$.

 If $\mathcal F$ has an isolated compact leaf $L$ then the vector fields  $\alpha_1^\sharp, ..., \alpha_k^\sharp$ must be tangent to $L$ as $L$ is equal to the stratum containing it. Thus  any basic 1-form vanishes along this compact leaf, a contradiction.
\end{proof}

\bibliography{sn-bibliography}

\end{document}